\documentclass[12pt,english]{article}
\usepackage[letterpaper]{geometry}
\usepackage{amsmath,amssymb,amsthm}
\usepackage{setspace}
\usepackage{babel}
\usepackage{float}

\theoremstyle{plain}
\newtheorem{theorem}{Theorem}
\newtheorem*{theorem*}{Theorem}
\newtheorem{lemma}{Lemma}
\newtheorem{corollary}{Corollary}
\newtheorem*{corollary*}{Corollary}
\newtheorem{definition}{Definition}
\newtheorem*{remark*}{Remark}

\def\mathscr{\mathfrak}

\newcommand{\ds}{\displaystyle}

\def\Ran{\mathrm{Ran\,}}

\def\llangle{\langle\!\langle}
\def\rrangle{\rangle\!\rangle}

\def\e{\varepsilon}

\def\s{\sigma}
\def\l{\lambda}

\def\T{\mathcal{T}}
\def\L{\mathcal{L}}
\def\S{\mathcal{S}}
\def\A{\mathcal{A}}
\def\X{\mathcal{X}}
\def\Y{\mathcal{Y}}

\def\F{\mathcal{F}}
\def\pfi{\varphi}
\def\D{\mathcal{D}}

\def\-{\backslash}
\def\ohat{\widehat{\otimes}}

\def\ds{\displaystyle}

\makeatletter

\date{}

\makeatother

\begin{document}

\title{Positive semigroups and abstract Lyapunov equations}

\author{Sergiy Koshkin\\
 Computer and Mathematical Sciences\\
 University of Houston-Downtown\\
 One Main Street, \#S705\\
 Houston, TX 77002\\
 e-mail: koshkins@uhd.edu}

\maketitle
\begin{abstract}\

We consider abstract equations of the form $\A x=-z$ on a locally convex space, where $\A$ generates a positive semigroup and $z$ is a positive element. This is an abstract version of the operator Lyapunov equation $A^*P+PA=-Q$ from control theory. It is proved that under suitable assumptions existence of a positive solution implies that $-\A$ has a positive inverse, and the generated semigroup is asymptotically stable. We do not require that $z$ is an order unit, or that the space contains any order units. As an application, we generalize Wonham's theorem on the operator Lyapunov equations with detectable right hand sides to reflexive Banach spaces.
\bigskip


\end{abstract}

\section{Introduction}\label{S1}

Lyapunov observed that stability of the matrix exponent $e^{tA}$ can be established by considering a matrix equation
$A^*P+PA=-Q$, where $P$ is the unknown and $Q$ is a given strictly positive definite matrix. This is the original Lyapunov equation. The equation is closely related to a one-parameter semigroup on the space of symmetric matrices 
$\T(t)P:=e^{tA^*}Pe^{tA}$, its generator is easily seen to be $\mathcal{A}P:=A^*P+PA$. This semigroup is a particular case of so--called implemented semigroups \cite[I.3.16]{EN}, \cite{GN}, \cite[3.4]{Kuh0}, we will refer to it as the Lyapunov semigroup. 

If $Q$ is strictly positive definite and $e^{tA}$ is exponentially stable then $P:=\int_0^\infty e^{tA^*}Qe^{tA}\,dt$ is a positive definite solution to the Lyapunov equation. A celebrated theorem of Lyapunov claims the converse: existence of a positive definite solution implies exponential stability \cite[8.7.2]{Lan}. This is most easily derived from the fact that $\T(t)$ is a positive semigroup on the space of symmetric matrices. Indeed, if we partially order this space by the cone of non-negative definite matrices then for $P\geq0$:
$$
\bigl(\T(t)Px,x\bigr)=\bigl(e^{tA^*}\,P\,e^{tA}x,x\bigr)=\bigl(P\,e^{tA}x,e^{tA}x\bigr)\geq0\,.
$$
So $\T(t)P\geq0$ for $P\geq0$ and $\T(t)$ preserves positivity. Noticing that exponential stabilities of $\T(t)$ and $e^{tA}$ are equivalent we see that Lyapunov's result gives us an equivalence between stability of a positive semigroup $\T(t)$ and positive solvability of the equation $\A P=-Q$ with its generator  $\A$. In addition, positive invertibility of $-\A$ is also equivalent to positive solvability of a single equation.

One difficulty in generalizing the Lyapunov theorem to infinite-dimensional spaces is that even if the underlying  semigroup $T(t)$, replacing $e^{tA}$, is $C_0$ the corresponding Lyapunov semigroup $\T(t)$ on the space of operators is not unless the underlying generator $A$ is bounded. Nonetheless, the Lyapunov equivalence was generalized to the case of Lyapunov semigroups on the space of bounded self-adjoint operators on Hilbert spaces \cite{Dat}, \cite[Lem.3]{Zb0},  and more generally on operator algebras \cite[4.2]{GN}. A crucial role in these generalizations is played by order units, elements strictly positive in a strong sense, which induce a special norm, the order unit norm, that ties together the partial order and the Banach structure. In particular, the identity operator on a Hilbert space is an order unit, and the order unit norm coincides with the spectral norm. However, in applications to the control theory the choice of space is often dictated by the nature of the problem, and it is often not a Hilbert space. 

We are primarily interested in Lyapunov semigroups on spaces of operators on reflexive Banach spaces, where no order units are present. This leads to an abstract setting of a vector space $\X$ ordered by a positive cone $\X^+$, a positive semigroup $\T(t)$ on it with the generator $\A$, and a positive element $z\geq0$ entering an abstract Lyapunov equation 
$\A x=-z$ on $\X$. Various conditions can be imposed on $\X$, $\A$ and $z$ to reproduce a version of the Lyapunov equivalence. The simplest result is obtained when $\X$ is an order unit space, $\A$ is the generator of a positive $C_0$ semigroup on $X$, and $z$ is an order unit \cite[Prop.4.1]{K}. We replace the condition that $z$ be an order unit with a weaker one, tailored to a specific semigroup $\T(t)$, that may hold even when no order units exist. This condition is inspired by the detectability condition of control theory \cite[3.6]{Wh}, and reduces to versions of it for Lyapunov semigroups. For positive $C_0$ semigroups it was first considered in \cite{K}. 

Analytically we work with positive semigroups on locally convex spaces that are not $C_0$, but retain enough properties of $C_0$ semigroups to make sense of generators and abstract Lyapunov equations, see Section \ref{S2}. Many cumbersome arguments of control theory, see e.g. \cite{Ben,CR,GvN,Meg,Zb}, originally developed for Lyapunov semigroups, extend to general positive semigroups and get streamlined in the process. Our main results, Theorems \ref{Lyapw} and \ref{Lyappi} of Section \ref{S3}, are generalizations of the Lyapunov triple equivalence to positive semigroups satisfying very mild analytic assumptions, but they are new even for $C_0$ semigroups on Banach spaces and for their adjoints, $C_0^*$ semigroups. In particular, we introduce a replacement for the order unit norm to prove a stronger version of the equivalence in Theorem \ref{Lyappi}. Aside from specific results our approach provides conceptual unification of positivity based arguments to prove stability.
 
As mentioned above, the Lyapunov semigroups on operator spaces are a particular case of implemented 
semigroups $\T(t)$ on spaces $\L(X,X^*)$ of bounded operators from a Banach space $X$ to its dual. We consider them in Section \ref{S4}, and show that $\T(t)$ are $C_0^*$ semigroups dual to  
$C_0$ semigroups $\T_*(t)$ on the projective tensor product $X\otimes_\pi X$. We then deal with the subspace of the symmetric operators $\L_s(X,X^*)$, where the Lyapunov semigroups live, particularly with the partial order and topologies on it. The main result of Section \ref{S4} is that the Lyapunov semigroups satisfy analytic assumptions in our abstract theorems when $X$ is reflexive.

Finally, in Section \ref{S5} our theorems are applied to the Lyapunov semigroups over reflexive Banach spaces. In particular, we generalize a finite-dimensional theorem of Wonham \cite[12.4]{Wh} that if the pair $(C,A)$ is exponentially detectable and the Lyapunov equation $A^*P+PA=-C^*C$ has a positive definite solution then the semigroup generated by $A$ is exponentially stable, see Theorem \ref{Wonh}. Several tests for verifying detectability are also given.

When $X=H=X^*$ is a Hilbert space a generalization of the Wonham's theorem is due to Zabczyk \cite[Lem.3]{Zb0}. Proving the Lyapunov equivalence over a Banach space $X$ was attempted in \cite{PH}, but the authors assumed existence of a continuous strongly positive definite quadratic form on $X$, which makes it isomorphic to a Hilbert space 
\cite[Thm.3]{Lin}. The Lyapunov equation in general Banach spaces is considered in \cite{GvN,vN2}, where the idea of using reproducing kernel Hilbert spaces is introduced, but without linking existence of solutions to stability of the semigroup. 

The Wonham's theorem is instrumental in proving existence of solutions to the matrix Riccati equation of optimal control, and monotone convergence to them of solutions to an iterative sequence of matrix Lyapunov equations \cite[12.6]{Wh}. We plan to apply our Theorem \ref{Wonh} to control systems in reflexive Banach spaces, a generalization to non-reflexive spaces is also of interest.

\section{Preliminaries}\label{S2}

In this section we fix terminology and notation used in the paper, and recall some basic results from the theory of continuous one-parameter semigroups. Let $\X$ be a real locally convex vector space with the dual space $\X'$. As usual, the duality pairing between $\X$ and $\X'$ is denoted $\langle\cdot,\cdot\rangle$, and $\sigma(\X,\X')$, $\sigma(\X',\X)$ denote the weak topologies on $\X$ and $\X'$ respectively. If $S$ is a densely defined linear operator on $\X$ then $S'$ is its adjoint with the natural domain \cite[IV.2]{Sch}. 
To talk about integrals we need one extra assumption about $\X$, see \cite[3.1.2]{BR}.
\begin{definition} A locally convex space $\X$ is said to have {\sl compactness of convex closures} if the 
closures of convex hulls of metrizable compacts in it are compact.
\end{definition}
\noindent Compactness of convex closures is a completeness assumption, and it follows in particular from sequential completeness of $\X$. Therefore, it is satisfied for weak* topology on duals to Banach spaces, and for the strong and the weak operator topologies. Moreover, it holds for the weak topology on any Banach space by the Krein-\u{S}mulian theorem, even though some of them, like $c_{\,0}$, are not weakly sequentially complete. Compactness of convex closures is a minimal assumption that guarantees existence of Pettis integrals for continuous $\X$-valued functions \cite[2.5.3]{BR}, \cite[3.26]{Rud}, \cite[3.3]{Ry}. 
\begin{definition} Let $\Omega$ be a compact metric space with finite positive Borel measure $\mu$, $\X$ be a locally convex space with compactness of convex closures, and $f:\Omega\to \X$ be continuous. Then there is a unique element $\int_\Omega f\,d\mu\in \X$, called the {\sl Pettis integral} of $f$ over $\Omega$, such that for all $\pfi\in \X'$:
$$
\langle\pfi,\int_\Omega f\,d\mu\rangle=\int_\Omega\langle\pfi,f\rangle\,d\mu\,.
$$
\end{definition}

In most existing theories of one--parameter semigroups on locally convex spaces the semigroups are either assumed to be locally equicontinuous \cite{Alb}, \cite[IX.2]{Yo}, or $\X$ is equipped with an additional Banach norm, in which they are assumed to grow no faster than exponentially in time \cite[3.1.2]{BR}, \cite{Kuh0}--\cite{Kun}, \cite{Sh1}. In both cases the semigroups have the property of locally uniform boundedness $C_b$ defined below, which is easier to verify and which suffices for our purposes. For instance, locally equicontinuous \cite{Alb} and bi--continuous semigroups \cite{Kuh0,Kuh} are always $C_b$, and integrable semigroups on norming dual pairs \cite{Kun0,Kun} are $C_b$ if they are continuous. Of course, $C_0$ semigroups on Banach spaces are also $C_b$. 
\begin{definition}\label{Cb}
A family of continuous operators $\T(t):\X\to \X$ is called a $C_b$ semigroup on $\X$ if

{\rm1)} $\T(0)=I$ and $\T(t+s)=\T(t)\T(s)$ for all $t,s\geq0$.

{\rm2)} $t\mapsto \T(t)x$ is continuous on $[0,\infty)$ for all $x\in \X$.

{\rm3)} $\ds{x\mapsto\sup_{0\leq t\leq a}p\big(\T(t)x\big)}$ is bounded on bounded subsets of $\X$ for all continuous semi-norms $p$ on $\X$, and for all $a>0$ ({\sl locally uniform boundedness}).

\noindent We call $\T(t)$  a $C_b^\sigma$ semigroup if it is a $C_b$ semigroup on $\X$ equipped with $\sigma(\X,\X')$ topology.
\end{definition}

As usual, we define the infinitesimal generator of a $C_b$ semigroup as
$$
\A x:=\lim_{t\to0+}\frac1t(\T(t)-I)x
$$
on $\mathcal{D}_\A$ where the limit exists, and the integrated semigroup as $\S(t)x:=\int_0^t\T(s)x\,ds$ on all of $\X$. 
The integral defining $\S(t)$ is a Pettis integral with $\Omega=[0,t]$ and $\mu$ the Lebesgue measure on it. What we mostly use are the properties of $\A$ and $\S(t)$ collected in the next Lemma. The proofs are similar to those in 
\cite{Kuh,Kun,Sh1}, one can check that locally uniform boundedness can replace norm estimates in the proofs.
\begin{lemma}\label{AS(t)} Let $\T(t)$ be a $C_b$ semigroup on $\X$ with the generator $\A$ and the integrated semigroup 
$\S(t)$. Then

{\rm(i)} $t\mapsto \S(t)x$ is continuously differentiable on $[0,\infty)$, and $\frac{d}{dt}\S(t)x=\T(t)x$ for all $x\in \X$;

{\rm(ii)} $\S(t)\X\subseteq\mathcal{D}_\A$, $\A\,\S(t)=\T(t)-I$ and $\mathcal{D}_\A$ is dense in $\X$;

{\rm(iii)} $x\mapsto \S(t)x$ is sequentially continuous;

{\rm(iv)} $\A\,\S(t)=\S(t)\A$ on $\mathcal{D}_\A$;

{\rm(v)} $\A$ is sequentially closed;

{\rm(vi)} $\int_0^t\S(\tau)x\,d\tau\in\mathcal{D}_\A$ for any $x\in \X$ and  
$\A\int_0^t\S(\tau)x\,d\tau=\bigl(\S(t)-tI\bigr)x$;

{\rm(vii)} $t\mapsto \S(t)x$ is locally uniformly bounded, and $\A\int_0^t\S(\tau)x\,d\tau=\int_0^t\S(\tau)\A x\,d\tau$ for 
$x\in\mathcal{D}_\A$.
\end{lemma}
\noindent Local equicontinuity is usually used to prove that $\S(t)$ is a continuous operator, but locally uniform boundedness only implies sequential continuity. Accordingly, the generator is only sequentially closed. The author of \cite{Sh1} asserts continuity of $\S(t)$ in a similar context, but the intended proof overlooks that convergent nets may be unbounded. If strong continuity of $\S(t)$ can be established by other means then unrestricted closure for $\A$ will also follow, see e.g. \cite{Kun0} where the Banach structure is used. However, unrestricted closure is more than one needs for our purposes.

\section{Stability of positive semigroups\\ and Lyapunov equations}\label{S3}

In this section we assume that $\A$ generates a positive $C_b$ semigroup and introduce suitable notions of stability to prove general versions of the Lyapunov equivalence. As is well known, there exist multiple inequivalent concepts of stability in infinite dimensions \cite{vNSW}. Accordingly, we prove two kinds of Lyapunov theorems. The first one assumes that bounded monotone sequences converge in $\X$, but only applies to $\sigma(\X,\X')$ continuous semigroups, the 
$C_b^\s$ semigroups. This result and its proof demostrate the idea behind the Lyapunov equivalence most clearly, but the type of stability it provides is often too weak for applications. For the second theorem we assume that $\X,\X'$ form a norming dual pair of Banach spaces, and require the semigroup to be compatible with the norm in some sense. The norm then serves as a replacement for the missing order unit norm. The concept of detectors, central in this section, as well as the ideas of proving main Theorems \ref{Lyapw} and \ref{Lyappi} come from inspecting positivity based arguments in control theory, see especially \cite{GvN} and \cite{Zb0}.

Let $\X$ be a locally convex space with compactness of convex closures. Assume that it is partially ordered by a closed proper cone $\X^+$, and let $\X'^+:=\{\pfi\in \X'\,|\, \langle\pfi,x\rangle\geq0\text{ for all }x\in \X^+\}$ be the dual cone as usual. We will always assume that $\X'^+$, but not necessarily $\X^+$, is {\sl generating} (a.k.a. reproducing), i.e. $\X'=\X'^+-\X'^+$. This is dictated by applications, see Section \ref{S4}. We wish to study the Lyapunov equation $\A x=-z$ for generators of positive $C_b$ semigroups on $\X$. If $z\geq0$ is appropriately chosen existence of a positive solution $x$ should imply inverse positivity of $-\A$ and asymptotic stability of its semigroup $\T(t)$. 

Recall that an element $e\in \X^+$ is called an order unit if for every $x\in\X$ there is $\lambda>0$ such that $-\l e\leq x\leq\l x$. One can then introduce a norm on $\X$, called the order unit norm \cite[A.2.7]{Clem}, by $||x||_e:=\inf\{\l>0\,|-\l e\leq x\leq\l e\}$. If $\X$ is complete in this norm it is a Banach space called an order unit space. In order unit spaces one can take $z=e$ to prove a strong form of the Lyapunov equivalence, see e.g. \cite[Prop.4.1]{K}. Note that 
$\X$ may contain order units without being an order unit space, e.g. if we equip an order unit space with a different topology, but we are mostly interested in cases where no order units are present at all.

We begin by specifying the type of stability and conditions on $z$ that deliver a Lyapunov type theorem for the $C_b^\s$ semigroups.
\begin{definition} Let $\T(t)$ be a positive $C_b^\s$ semigroup on $\X$. It is called {\sl weakly $L^1$ stable on $\X^+$} if $\int_0^\infty\langle\pfi,\T(t)x\rangle\,dt<\infty$ for all $\pfi,x\geq0$. An element $z\in \X^+$ is called a {\sl weak $L^1$ detector} if for every $\pfi\in \X'^+$:
$$
\int_0^\infty\langle\pfi,\T(t)z\rangle\,dt<\infty\implies\int_0^\infty\langle\pfi,\T(t)x\rangle\,dt<\infty\text{ for all }x\in \X^+.
$$
\end{definition}
\noindent Since we do not assume that $\X^+$ is generating our weak $L^1$ stability on $\X^+$ may be weaker than the usual weak $L^1$ stability on $\X$ \cite{vNSW}. The point is that we are interested in tracking $\T(t)$ only on the positive cone. By definition, if $z$ is a weak $L^1$ detector then $\langle\pfi,\T(t)z\rangle$ displays the correct asymptotic behavior for every $\pfi\geq0$, i.e. $z$ literally detects weak $L^1$ stability. For positive 
$C_0$ semigroups weak $L^1$ detectors were considered in \cite{K}. For Lyapunov semigroups over Hilbert spaces weak $L^1$ detection will reduce to a well-known notion of control theory, see Section \ref{S5}. We will discuss verification of this property at length later, for now note that if $\X$ contains order units then any of them is a weak $L^1$ detector for any positive semigroup. Moreover, if $z$ is a detector then so is $z+a$ for any $a\geq0$. 
\begin{theorem}\label{Lyapw} Let $\T(t)$ be a positive $C_b^\s$ semigroup on $\X$ with the generator $\A$ and a weak 
$L^1$ detector $z$. Assume additionally that topologically bounded monotone sequences converge in $\X$. 
Then the following conditions are equivalent:

{\rm(i)} $\A x=-z$ has a positive solution $x\in\D_\A\cap\X^+$;

{\rm(ii)} $\T(t)$ is weakly $L^1$ stable on $\X^+$;

{\rm(iii)} $\A$ is algebraically invertible on $\X^+$ and $-\A^{-1}\geq0$.
\end{theorem}
\begin{proof}${\rm(i)}\Rightarrow{\rm(ii)}$ Since $x\in\D_\A$ we have by Lemma \ref{AS(t)}(ii),(iv) 
that $\S(t)\A x=\bigl(\T(t)-I\bigr)x=-\S(t)z=-\int_0^t\T(s)z\,ds$. Therefore, for any $\pfi\in \X'^+$
$$
\int_0^t\langle\pfi,\T(s)z\rangle\,ds=\langle\pfi,x-\T(t)x\rangle\leq\langle\pfi,x\rangle
$$
since $\T(t)x\geq0$. This implies $\int_0^\infty\langle\pfi,\T(t)z\rangle\,dt<\infty$ for all $\pfi\in \X'^+$, and 
$\T(t)$ is weakly $L^1$ stable on $\X^+$ since $z$ is a detector.

${\rm(ii)}\Rightarrow{\rm(iii)}$ Let $x\in \X^+$ and $\pfi\in \X'^+$, then by definition of $\S(t)$
$$
\langle\pfi,\S(t)x\rangle=\int_0^t\langle\pfi,\T(s)x\rangle\,ds\leq\int_0^\infty\langle\pfi,\T(s)x\rangle\,ds<\infty.
$$
Since $\X'^+$ is generating there are $\pfi^\pm\in \X'^+$ such that $\pfi=\pfi^+-\pfi^-$ for any $\pfi\in \X'$. It follows that $|\langle\pfi,\S(t)x\rangle|$ is bounded for any $\pfi\in \X'$. This means that $\{\S(t)x,\,t\geq0\}$ is $\sigma(\X,\X')$ bounded for all $x\in \X^+$. Since $\S(t)x=\int_0^t\T(s)x\,ds$ and $\T(s)\geq0$ we see that $\S(t)x$ is monotone increasing in $t$. By the assumption about monotone sequences, there is a limit $\S(t)x\xrightarrow[t\to\infty]{}y$ and we define 
$\S_\infty x:=y$ on $\X^+$. Obviously, $\S_\infty\geq0$.

Since $\S(t)x\xrightarrow[t\to\infty]{}\S_\infty x$ we have $\frac1t\S(t)x\xrightarrow[t\to\infty]{}0$, so 
$\frac1t\int_0^t\S(\tau)\,d\tau\xrightarrow[t\to\infty]{}\S_\infty x$ by the usual property of Ces\`aro limits. 
By Lemma \ref{AS(t)}(vi), $\frac1t\int_0^t\S(\tau)\,d\tau\in\D_\A$ and $\A\,\frac1t\int_0^t\S(\tau)\,d\tau=\frac1t\S(t)x-x$. 
Passing to limit one obtains $\S_\infty x\in\D_\A$ and $\A \S_\infty x=-x$ since $\A$ is sequentially closed.

If $x\in\D_\A$ and $-\A x\geq0$ then as above $\S(t)\bigl(-\A x\bigr)\xrightarrow[t\to\infty]{}\S_\infty\bigl(-\A x\bigr)$, 
$\frac1t\S(t)\bigl(-\A x\bigr)\xrightarrow[t\to\infty]{}0$ and $\frac1t\int_0^t\S(\tau)\bigl(-\A x\bigr)\,d\tau\xrightarrow[t\to\infty]{}\S_\infty\bigl(-\A x\bigr)$. By Lemma \ref{AS(t)}(vi),(vii), for $x\in\D_\A$ we have on the other hand:
$$
\frac1t\int_0^t\S(\tau)\bigl(-\A x\bigr)\,d\tau=-\frac1t\A\int_0^t\S(\tau)x\,d\tau
=-\frac1t\bigl(\S(t)-tI\bigr)x=-\frac1t\S(t)x+x\xrightarrow[t\to\infty]{}x\,,
$$
so $\S_\infty\bigl(-\A x\bigr)=x$. Thus, $\S_\infty\geq0$ is the algebraic inverse of $-\A$ on $\X^+$.

${\rm(iii)}\Rightarrow{\rm(i)}$ is obvious since $z\in \X^+$ by definition.
\end{proof}
\noindent The proof of ${\rm(ii)}\Rightarrow{\rm(iii)}$ would be simpler if weak $L^1$ stability implied weak stability, i.e. $\T(t)x\xrightarrow[t\to\infty]{}0$ weakly. Then we could use that
$\A \S(t)=\T(t)x-x$ directly to establish invertibility instead of resorting to Ces\`aro limits. However, weak convergence of $\frac1t\S(t)x$ to $0$ may not imply the same for $\T(t)x$ even for $C_0$ semigroups in Banach spaces \cite[Prop.4.4]{EFNS}.

To illustrate the theorem we will apply it to $C_0^*$ semigroups, the weak* continuous semigroups on spaces dual to Banach spaces. If $\T(t)$ is such a semigroup on $\X$ then the adjoint semigroup $\T^*(t)$ on the Banach dual space 
$\X^*$ leaves the pre-dual $\X_*\subseteq \X^*$ invariant, and $\T(t)$ is adjoint to the restriction $\T_*(t)$ of $\T^*(t)$ to $\X_*$. This pre-dual semigroup is a $C_0$ semigroup on $\X_*$ \cite[3.1]{BR}. Recall that a positive cone is called normal if order bounded sets are topologically bounded. It is known that if $\X^+$ is normal and generating then so is 
$\X^{*+}$ \cite[A.2.4]{Clem}, and any positive operator from a space with a generating cone to a space with a normal cone is norm bounded \cite[A.2.11]{Clem}.
\begin{corollary}\label{Lyapw*} Let $\T(t)$ be a positive $C_0^*$ semigroup on a space dual to a Banach space with the generator $\A$ and a weak* $L^1$ detector $z$. Assume that $\X^+$ is normal and generating. Then the following conditions are equivalent:

{\rm(i)} $\A x=-z$ has a positive solution $x\in\D_\A\cap \X^+$;

{\rm(ii)} $\T(t)$ is weakly* $L^1$ stable, i.e. $\int_0^\infty|\langle\pfi,\T(t)x\rangle|\,dt<\infty$ for all 
$\pfi\in \X_*^+$ and $x\in \X^+$;

{\rm(iii)} $\A$ has a bounded inverse and $-\A^{-1}\geq0$.
\end{corollary}
\begin{proof}
By definition, $C_0^*$ semigroups are exactly the $\sigma(\X,\X_*)$ continuous semigroups, moreover they are $C_b^\sigma$ due to norm estimates on $\T_*(t)$. By the Banach-Alaoglu theorem, bounded subsets of $\X$ are weak* precompact implying that bounded monotone sequences weakly* converge in $\X$. Thus, the assumption on monotone sequences in Theorem \ref{Lyapw} is automatically satisfied for Banach $\X$ with weak* topology. Since $\X^+,\X^{*+}$ are both generating weak* $L^1$ stability on $\X^+$ is equivalent to the usual one on $\X$ because $\pfi, x$ can be decomposed into positive and negative parts. For the same reason $-\A^{-1}$ extends to an everywhere defined positive linear operator on $\X$. Since $\X^+$ is normal and generating this operator is norm bounded, the corollary now follows directly from Theorem \ref{Lyapw}.
\end{proof}
\noindent The implication ${\rm(ii)}\Rightarrow{\rm(iii)}$ without positivity of $-\A^{-1}$ holds for any $C_0^*$ semigroup, it is usually stated in terms of its pre-dual $C_0$ semigroup $\T_*(t)$ \cite[7.3]{Clem}. If in addition $\T(t)\geq0$ then inverse positivity also follows from the Laplace transform formula for the resolvent. The converse, however, requires that $\T(t)$ admit a weak* $L^1$ detector. In particular, it always holds in order unit spaces that are dual to a Banach space, e.g. $L^\infty(\mu)$ and spaces of bounded self-adjoint operators on Hilbert spaces.

In general, Theorem \ref{Lyapw} does not guarantee any regularity for $\A^{-1}$. Even its existence on $\X^+$ required an extra assumption about convergence of bounded monotone sequences that was instrumental in the proof. This is because weak $L^1$ stability by itself is too weak to impart $\int_0^\infty \T(t)\,dt$ with any regularity. We can not even apply the theorem directly to all Banach spaces, because in some of them, like $c_0$, bounded monotone sequences may not converge. We saw one way to improve on this in Corollary \ref{Lyapw*}, but it is too special. Another way is to impose stronger analytic conditions on $\X$ accompanied by a stronger form of stability. 

Such stronger conditions will be imposed using norming dual pairs \cite[3.1.2]{BR}, \cite{Kun0,Kun,Sh1}. 
\begin{definition} Two Banach spaces $\X,\Y$ are called a {\sl norming dual pair} if $\Y$ is a closed subspace of the Banach dual $\X^*$, which is norming, i.e. $\|x\|=\sup\{|\langle\pfi,x\rangle|\,\big|\,\pfi\in\Y,\,\pi(\pfi)\leq1\}$ for all $x\in\X$. 
\end{definition}
\noindent From now to the end of this section we assume that $\X,\X'$ form a partially ordered norming dual 
pair, and denote the Banach norms on both spaces $\pi$. It follows from the definition that the norm topology on $\X$ is stronger than the original locally convex topology, and linear operators on $\X$ are norm bounded if and only if they are $\s(\X,\X')$ continuous \cite[1.2]{Kun0}. The norming condition also implies the following natural inequality for Pettis integrals:
\begin{equation}\label{NormPettis}
\pi\Big(\int_\Omega f\,d\mu\Big)
=\sup_{\pi(\pfi)\leq1}|\langle\pfi,\int_\Omega f\,d\mu\rangle|
=\sup_{\pi(\pfi)\leq1}|\int_\Omega \langle\pfi,f\rangle\,d\mu|
\leq\int_\Omega\sup_{\pi(\pfi)\leq1}|\langle\pfi,f\rangle|\,d\mu
=\int_\Omega\pi(f)\,d\mu\,.
\end{equation}
Remarkably, in the presence of order units weak $L_1$ stability often implies exponential stability (see Corollary \ref{Lyapord}), and we will be aiming at a similar result when order units are not available. The following compatibility  condition is meant to make the $\pi$ norms serve as a suitable generalization of the order unit norm $||\cdot||_e$ and its dual. 
\begin{definition} Let $\X,\X'$ be a partially ordered norming dual pair and  $\T(t)$ be a positive $C_b$ semigroup on 
$\X$. We call $\T(t)$ {\sl norm compatible or $\pi$-compatible} if for every $\pfi\in \X'^+$ the function $t\mapsto\pi\bigl(\T'(t)\pfi\bigr)$ is continuous, and we call it {\sl  $L^1$ $\pi$--stable} if $\int_0^\infty\pi\bigl(\T'(t)\pfi\bigr)\,dt<\infty$ for all $\pfi\in \X'^+$. 
\end{definition}
\noindent Of course, if the adjoint semigroup $\T^*(t)$ on $\X^*$ restricts to a $C_0$ semigroup on $\X'$ it is all the more norm compatible, but example of spaces with order units in weak* topology shows that $\pi$-compatibility is a weaker property than strong continuity of $\T'(t)$. Indeed, $\pi\bigl(\T'(t)\pfi\bigr)=\langle\pfi,\T(t)e\rangle$ for positive $\pfi$ when $\pi$ is the order unit norm, so $\pi\bigl(\T'(t)\pfi\bigr)$ is continuous even if $\T'(t)$ is only a $C_0^*$ semigroup. 

The requirement that $\X'$ is a closed subspace of $\X^*$, although technically convenient, can be too restrictive when selecting detectors. We will need more flexibility for dealing with the Lyapunov semigroups.
\begin{definition} An element $z\in \X^+$ is called an {\sl $L^1$ $\pi$-detector relative to $\F$}, where $\F$ is some subset of $\X'^+$, if for every $\pfi\in \F$:
\begin{equation*}\label{L1Detector}
\int_0^\infty\langle\pfi,\T(t)z\rangle\,dt<\infty\implies\int_0^\infty\pi\Bigl(\T'(t)\pfi\Bigr)\,dt<\infty\,.
\end{equation*}
It is called simply an $L^1$ $\pi$-detector when $\F=\X'^+$.

A subset $\F\subset \X'^+$ is called an {\sl $L^1$ $\pi$--stability subspace} for a class of semigroups if stability on $\F$ implies stability on $\X'^+$ for every semigroup in the class. 
\end{definition}
\noindent In our examples $\F$ will be a dense subcone of $\X'^+$, there was no need for it in the case of weak detection because one could choose $\X'$ much more freely. Since $\langle\pfi,\T(t)x\rangle\leq\pi(x)\,\pi\bigl(\T'(t)\pfi\bigr)$ it is immediate that the $\pi$--detection (relative to $\X'^+$) is stronger than the weak detection. It is not clear however, that $\pi$-detectors exist in the absense of order units. The following property can be directly verified in applications via a priori estimates \cite[3.27]{CR}, and turns out to imply $\pi$--detection. 
\begin{definition}\label{dpiobs} Let $\T(t)$ be a positive $\pi$-compatible $C_b$ semigroup on $\X$. 
An element $z\in \X^+$ is called a {\sl final $\pi$-observer at $t_0>0$ relative to $\F$} if there is $\e>0$ such that for every $\pfi\in \F$:
\begin{equation*}\label{FinObserver}
\int_0^{t_0}\langle\pfi,\T(s)z\rangle\,ds\geq\e\,\pi\Bigl(\T'(t_0)\pfi\Bigr)\,.
\end{equation*}
\end{definition}
\noindent The point of the name is that the final state $\T'(t_0)\pfi$ is recoverable from observing 
$\langle\pfi,\T(s)z\rangle$ for $0\leq s<t_0$. The proof of the next theorem is based on the idea from \cite[Thm.4.5]{Meg}.
\begin{theorem}\label{piobs} Relative to any $\T'$--invariant subset $\F\subset \X'^+$ every final $\pi$-observer is an 
$L^1$ $\pi$--detector. 
\end{theorem}
\begin{proof}
Pick any $\pfi\geq0$ and consider $\psi:=\T'(t-t_0)\pfi\geq0$ for $t\geq t_0$.  Since $\pfi\in\F$ and $\F$ is $\T'$--invariant also $\psi\in\F$, so applying the definition to $\psi$ and a final $\pi$-observer $z$ at $t_0$:
$$
\int_0^{t_0}\langle\pfi,\T(s+t-t_0)z\rangle\,ds
=\int_0^{t_0}\langle\psi,\T(s)z\rangle\,ds
\geq\e\,\pi\Bigl(\T'(t_0)\psi\Bigr)
=\e\,\pi\Bigl(\T'(t)\pfi\Bigr)\,.
$$
Therefore, for $\tau\geq t_0$
\begin{multline}\label{piT}
\int_0^{\tau}\pi\Bigl(\T'(t)\pfi\Bigr)\,dt
=\int_0^{t_0}\pi\Bigl(\T'(t)\pfi\Bigr)\,dt+\int_{t_0}^{\tau}\pi\Bigl(\T'(t)\pfi\Bigr)\,dt\\
\leq\int_0^{t_0}\pi\Bigl(\T'(t)\pfi\Bigr)\,dt
+\frac1{\e}\int_{t_0}^{\tau}\int_0^{t_0}\langle\pfi,\T(s+t-t_0)z\rangle\,ds\,dt\,.
\end{multline}
The pairing $\langle\pfi,\T(s+t-t_0)z\rangle$ is jointly continuous in $s$ and $t$ and we can switch the order of integration in the second term of \eqref{piT} yielding
$$
\int_{t_0}^{\tau}\int_0^{t_0}\langle\pfi,\T(s+t-t_0)z\rangle\,ds\,dt
=\int_{t_0}^{\tau}\int_s^{s+\tau-t_0}\langle\pfi,\T(t)z\rangle\,dt\,ds
\leq t_0\int_0^\infty\langle\pfi,\T(t)z\rangle\,dt\,.
$$
Since this estimate holds for all $\tau\geq t_0$ and all integrands are positive we can set $\tau=\infty$ in 
\eqref{piT}. Then, provided $\int_0^\infty\langle\pfi,\T(t)z\rangle\,dt<\infty$, we have 
$$
\int_0^\infty\pi\Bigl(\T'(t)\pfi\Bigr)\,dt\leq\int_0^{t_0}\pi\Bigl(\T'(t)\pfi\Bigr)\,dt\,
+\frac{t_0}{\e}\int_0^\infty\langle\pfi,\T(t)z\rangle\,dt<\infty\,,
$$
so by definition $z$ is an $L^1$ $\pi$-detector.
\end{proof}

In the proof of our main result we will need a generalization of the Datko-Pazy theorem given by van Neerven. The classical Datko-Pazy theorem \cite[II.1.2.2]{Ben}, \cite[Prop.9.4]{Clem} states that if $\T(t)$ is a $C_0$ semigroup and $\int_0^\infty\|\T(t)x\|^p\,dt$ for all $x\in\X$ and some $p\geq1$ then $\|\T(t)\|\leq Me^{-\varepsilon t}$ with $M,\e>0$, i.e. $\T(t)$ is exponentially stable. Van Neerven showed that the strong continuity of $\T(t)x$ can be replaced with local boundedness for the conclusion to follow, see Theorem 4.2 and Remark 4.3 in \cite{vN1}.
\begin{theorem}\label{Lyappi} Let $\T(t)$ be a positive $\pi$-compatible $C_b$ semigroup on a norming dual pair $\X,\X'$ with the generator $\A$ and an $L^1$ $\pi$-detector $z$ relative to $\F$, where $\F$ is a stability subset. Assume additionally that $\X'^+$ is generating. Then the following conditions are equivalent:

{\rm(i)} $\A x=-z$ has a positive solution $x\in\D_\A\cap \X^+$;

{\rm(ii)$'$} $\T(t)$ is $L^1$ $\pi$-stable on $\X^+$;

{\rm(ii)} $\T(t)$ is exponentially $\pi$-stable;

{\rm(iii)} $\A$ has a $\pi$-bounded inverse on $\X$ and $-\A^{-1}\geq0$.
\end{theorem}
\begin{proof} ${\rm(i)}\Rightarrow{\rm(ii)'}$ Proof that $\int_0^\infty\pi\bigl(\T'(t)\pfi\bigr)\,dt<\infty$ for all $\pfi\in\F$ is analogous to the proof of Theorem \ref{Lyapw} with $L^1$ $\pi$-detection replacing weak $L^1$ detection. Since $\F$ is a stability subset it follows that $\T(t)$ is $L^1$ $\pi$-stable $\X^+$.

${\rm(iii)}\Rightarrow{\rm(i)}$ is obvious since $z\geq0$.

${\rm(ii)}'\Rightarrow{\rm(ii)}$ Since $\X'^+$ is generating we have for any $\pfi\in\X'$ that $\pfi=\pfi^+-\pfi^-$, where $\pfi^{\pm}\in\X'^+$, and $\pi\bigl(\T'(t)\pfi\bigr)\leq\pi\bigl(\T'(t)\pfi^+\bigr)+\pi\bigl(\T'(t)\pfi^-\bigr)$. By $\pi$-compatibility, the right hand side is continuous in $t$, and by $\pi$-stability, it belongs to $L^1$. It follows from \cite[Thm.4.2]{vN1} that $t\mapsto\pi\bigl(\T'(t)\pfi\bigr)$ is measurable and $\int_0^\infty\pi\bigl(\T'(t)\pfi\bigr)\,dt<\infty$ for all $\pfi\in\X'$. By the van Neerven's Datko-Pazy theorem, $\T'(t)$ is exponentially $\pi$-stable. Dualizing, we have the same for $\T(t)$.

${\rm(ii)}\Rightarrow{\rm(iii)}$ Suppose $\pi\big(\T(t)x\big)\leq Me^{-\varepsilon t}\pi(x)$ and consider the Pettis integral $\S(t)x=\int_0^{t}\T(s)x\,ds$. By the inequality \eqref{NormPettis},
\begin{equation}\label{piCauchy}
\pi\Bigl(\S(t)x-\S(\tau)x\Bigr)\leq\int_{\tau}^t\pi\Bigl(\T(s)x\Bigr)\,ds
\leq\frac{M}{\e}\,(e^{-\e\tau}-e^{-\e t})\,\pi(x)
\xrightarrow[t,\tau\to\infty]{}0\,.
\end{equation}
Therefore, there is $y\in\X$ such that $\S(t)x\xrightarrow[t\to\infty]{}y$ by norm, and we set $\S_\infty x:=y$. Setting $\tau=0$ in \eqref{piCauchy} and taking $t\to\infty$ we have $\pi\bigl(\S_\infty x\bigr)\leq\,\frac{M}{\e}\,\pi(x)$, so $\S_\infty$ is a norm bounded linear operator on $\X$. By Lemma \ref{AS(t)}(ii), $\A\,\S(t)x=\T(t)x-x$ for all $x\in \X$, and $\T(t)x\xrightarrow[t\to\infty]{}0$ due to exponential stability. By Lemma \ref{AS(t)}(v), $\A$ is sequentially closed implying that 
$\S_\infty x\in\D_\A$ and $\A\S_\infty x=-x$. Analogously, for $x\in\D_\A$ we have 
$\S(t)\A x=\T(t)x-x$ by Lemma \ref{AS(t)}(iv), and $\S_\infty \A x=-x$ by passing to limit $t\to\infty$. Thus, 
$-\A^{-1}=\S_\infty\geq0$.
\end{proof}
\noindent 
Unlike Theorem \ref{Lyapw} this theorem can be applied directly to $C_0$ semigroups on Banach spaces
since only strong completeness is required. Of course, we still make an additional assumption that 
$t\mapsto\|\T^*(t)\pfi\|$ is continuous for $\pfi\in \X^{*\,+}$. This assumption is automatically satisfied if $\X$ is reflexive, but it may hold even if $\X$ is not. In particular, it holds in all order unit spaces for which 
Theorem \ref{Lyappi} greatly simplifies.
\begin{corollary}\label{Lyapord} Let $\X$ be an order unit space with the norm $\|\cdot\|_e$, and let $\|\cdot\|_e$ also denote the dual norm on $\X^*$. Let $\T(t)$ be a positive $C_0$ semigroup on $\X$ with the generator $\A$ and an $L^1$ 
$\|\cdot\|_e$-detector $z$. Then the following conditions are equivalent:

{\rm(i)} $\A x=-z$ has a positive solution $x\in\D_\A\cap \X^+$;

{\rm(ii)$^\mathrm{w}$} $\T(t)$ is weakly $L^1$ stable on $\X^+$;

{\rm(ii)} $\T(t)$ is exponentially stable;

{\rm(iii)} $\A$ has a bounded inverse on $\X$ and $-\A^{-1}\geq0$.
\end{corollary}
\begin{proof} For $\pfi\in \X^{*+}$ we have by the properties of the dual order unit norm that $\|\T^*(t)\pfi\|_e=\langle \T^*(t)\pfi,e\rangle=\langle\pfi,\T(t)e\rangle$. Therefore, $t\mapsto\|\T^*(t)\pfi\|_e$ is continuous and we can apply Theorem \ref{Lyappi}. Since $\|\T^*(t)\pfi\|_e=\langle\pfi,\T(t)e\rangle$ exponential stability of $\T(t)$, and hence of $\T^*(t)$, is equivalent to weak $L^1$ stability of $\T(t)$ on $\X^+$.
\end{proof}
\noindent  Even though only $C_0$ semigroups are involved in the statement of Corollary \ref{Lyapord} the proof requires consideration of $\T^*(t)$, which may not be a $C_0$ semigroup on $\X^*$. Also note an intriguing analogy with Hilbert spaces, where  weak $L^1$ stability also implies exponential stability for all $C_0$ semigroups \cite{Ws}.

\section{Implemented and Lyapunov semigroups}\label{S4}

In this section we establish analytic properties of the Lyapunov semigroups over reflexive Banach spaces required to apply to them the results of Section \ref{S3}. Since the Lyapunov semigroups are a particular case of implemented semigroups we start by reviewing those. The tensor product spaces and semigroups then appear naturally as their duals. The duality between the spaces of operators and tensor product spaces is the main technical tool of this section. 

An implemented semigroup is induced by the implementing semigroups denoted $T(t),V(t)$ on a Banach space $X$. The implemented semigroup $\T(t)P:=V^*(t)PT(t)$ then lives on the space $\L(X,X^*)$ of bounded operators from $X$ to $X^*$ \cite[3.4]{Kuh0}. Recall that the algebraic tensor product $X\otimes X$ is the space of finite linear combinations of monomials $x\otimes y$ with $x,y\in X$. The tensor product semigroup on the algebraic tensor product $X\otimes X$ is specified by setting $\T_*(t)(x\otimes y):=T(t)x\otimes V(t)y$. There is a natural duality pairing between $\L(X,X^*)$ and $X\otimes X$ given by $\llangle P,x\otimes y\rrangle:=\langle Px,y\rangle$ on monomials, here $\langle \cdot,\cdot\rangle$ is the duality pairing between $X^*$ and $X$. In general, $V^*(t)$ may only be a $C_0^*$ semigroup, but if $X$ is reflexive it is a $C_0$ semigroup on $X^*$ \cite[3.1.8]{BR}, \cite[IX.1]{Yo}. 

We now recall the definition of the projective tensor norm \cite[2.2.1]{Ry}, which will serve as the $\pi$ norm of Section \ref{S3}.
\begin{definition}\label{pinorm} Let $X$ and $Y$ be Banach spaces and $X\otimes Y$ be their algebraic tensor product. Given $\rho\in X\otimes Y$ its {\sl projective tensor norm} is defined by 
$$
\textstyle{\pi(\rho):=\inf\{\sum_i\|x_i\|\|y_i\|\,\Big|\,\rho=\sum_ix_i\otimes y_i,\,x_i\in X,\,y_i\in Y\}\,.}
$$
The {\sl projective tensor product $X\otimes_\pi Y$} is the completion of $X\otimes Y$ in this norm. 
\end{definition}
\noindent If $X=Y=H$ then $X\otimes_\pi Y$ is isomorphic as a Banach space to the space of trace class operators 
$\L_1(H)$, and the projective tensor norm is equivalent to the trace norm $\text{tr}(P^*P)^{\frac12}$. There is a well-known duality $\L(H)\simeq\L_1(H)^*$ \cite[VI.6]{RS}, which generalizes to the {\sl projective duality} $\L(X,X^*)\simeq(X\otimes_\pi X)^*$ \cite[2.2.2]{Ry}. Thus, $X\otimes_\pi X$ is the Banach pre-dual to $\L(X,X^*)$ and the $\s(\L(X,X^*),X\otimes_\pi X)$ topology is the weak* topology on $\L(X,X^*)$. 

It follows from Definition \ref{pinorm} that $\pi(x\otimes y)=\|x\|\|y\|$, so the dual norm $\pi$ on $\L(X,X^*)$ is just the induced operator norm $\|\cdot\|$, in the Hilbert case $X=H$ it coincides with the spectral norm on $\L(H)$. Being a pre-dual, $X\otimes_\pi X$ is a closed subspace of the Banach dual to $\L(X,X^*)$. Since $X\otimes X$ is already norming for $\L(X,X^*)$ the spaces $\X=\L(X,X^*)$, $\X'=X\otimes_\pi X$ form a norming dual pair. 

Aside from the weak* and the norm topology we will also need the weak operator topology on $\L(X,X^*)$, which can be identified with $\s(\L(X,X^*),X\otimes X^{**})$. When $X$ is reflexive it is the same as $\s(\L(X,X^*),X\otimes X)$ and then clearly is the weakest of the three. However, a simple argument, analogous to the one for 
$\L(H)$ \cite[II.2.5]{Tak}, shows that it coincides with the weak* topology on bounded subsets of $\L(X,X^*)$. When $X$ is not reflexive it is unclear if the norm closure of $X\otimes X^{**}$ in the Banach dual to $\L(X,X^*)$ is 
$X\otimes_\pi X^{**}$, or whether it admits any other explicit description. Hence, to make use of the projective duality and the weak operator topology we will assume $X$ to be reflexive.

The next theorem seems to be a new observation on the relationship between the implemented and the tensor product semigroups. We will use it to establish $\pi$--compatibility of the Lyapunov semigroups over Banach spaces. 
The proof makes essential use of the Grothendieck representation theorem for elements of 
$X\otimes_\pi X$ as sums $\sum_{n=1}^\infty a_n\,x_n\!\otimes y_n$ with bounded sequences $x_n, y_n\in X$, and a summable numerical sequence $a_n$ \cite[III.6.4]{Sch}. This theorem is valid even without the reflexivity assumption on $X$.
\begin{theorem}\label{C0pi} Let $T(t)$, $V(t)$ be $C_0$ semigroups on a Banach space $X$.
Then the tensor product semigroup $\T_*(t)(x\otimes y):=T(t)x\otimes V(t)y$ extends to a $C_0$ semigroup on 
$X\otimes_\pi X$, and the implemented semigroup $\T(t)P:=V^*(t)\,P\,T(t)$ is its adjoint $C_0^*$ semigroup on $\L(X,X^*)$. 
\end{theorem}
\begin{proof} Let $\rho=\sum_{i=1}^N x_i\otimes y_i\in X\otimes X$, then 
\begin{multline*}
\textstyle{\pi\Big(\T_*(t)\rho\Big)\leq\sum_{i=1}^N\pi\Big(T(t)x_i\otimes V(t)y_i\Big)}\\
\textstyle{=\sum_{i=1}^N\|T(t)x_i\|\,\|V(t)y_i\|\leq\|T(t)\|\,\|V(t)\|\sum_{i=1}^N\|x_i\|\,\|y_i\|\,.}
\end{multline*}
Passing to the infimum on both sides over all possible decompositions of $\rho$ we obtain the estimate
$\pi\big(\T_*(t)\rho\big)\leq\|T(t)\|\,\|V(t)\|\pi(\rho)$ for all $\rho\in X\otimes X$. Therefore, $\T_*(t)$ extends to a semigroup of bounded linear operators on the projective completion $X\otimes_\pi X$.

We prove strong continuity of $\T_*(t)$ on monomials first:
\begin{multline*}
\pi\Bigr(\T_*(t)(x\otimes y)-x\otimes y\Bigl)=\pi\Bigr(T(t)x\otimes V(t)y-x\otimes y\Bigl)\\
\leq\pi\Bigr(\bigr(T(t)x-x\bigl)\otimes V(t)y\Bigl)+\pi\Bigr(x\otimes\bigr(V(t)y-y\bigl)\Bigl)
=\|T(t)x-x\|\,\|V(t)\|\,\|y\|+\|V(t)y-y\|\,\|x\|\,.
\end{multline*}
Since $T(t)$, $V(t)$ are strongly continuous and $\|V(t)\|$ is locally bounded the last expression 
tends to $0$ when $t\to0$.

Now let $\rho\in X\otimes_\pi X$ be arbitrary. By the Grothendieck representation theorem \cite[III.6.4]{Sch}, 
$\rho=\sum_{n=1}^\infty a_n\,x_n\!\otimes y_n$, where $x_n,\,y_n$ are bounded and $\sum_{n=1}^\infty|a_n|<\infty$. 
It follows from the above estimate that
$$
\pi\Bigr(\T_*(t)\rho-\rho\Bigl)
\leq\sum_{n=1}^\infty|a_n|\,\Bigr(\|T(t)x_n-x_n\|\,\|V(t)\|\,\|y_n\|+\|V(t)y_n-y_n\|\,\|x_n\|\Bigl)\,.
$$
When $t\to0$ the two expressions in parentheses tend to $0$ for every $n$, they are also uniformly bounded since 
$x_n$, $y_n$ are bounded and $\|T(t)\|$, $\|V(t)\|$ are locally bounded. Since $a_n$ is summable the Lebesgue dominated convergence theorem implies that $\pi\Bigr(\T_*(t)\rho-\rho\Bigl)\xrightarrow[t\to0]{}0$. Thus, $\T_*(t)$ is strongly continuous on $X\otimes_\pi X$.

Finally, for $P\in\L(X,X^*)$ one has
\begin{multline*}
\llangle\T_*(t)^*P,x\otimes y\rrangle=\llangle P,\T_*(t)\bigl(x\otimes y\bigr)\rrangle
=\langle PT(t)x,V(t)y\rangle\\
=\langle V^*(t)PT(t)x,y\rangle=\llangle V^*(t)PT(t),x\otimes y\rrangle=\llangle\T(t)P,x\otimes y\rrangle\,.
\end{multline*}
Since $\T_*(t)^*$, $\T(t)$ are bounded and linear combinations of monomials are dense in $X\otimes_\pi X$ 
this implies $\T_*(t)^*=\T(t)$.
\end{proof}

Since $X$ is reflexive we identify $X$ and $X^{**}$ via the canonical embedding. Then the adjoint $P^*$ of 
$P\in\L(X,X^*)$ belong to the same space, and one can talk about symmetric and positive definite operators. 
\begin{definition} Let $X$ be a reflexive Banach space. An operator $P\in\L(X,X^*)$ is called symmetric if $P=P^*$ and positive definite if $\langle Px,x\rangle\geq0$ for all $x\in X$. The space of symmetric operators is denoted
$\L_s(X,X^*)$ and the cone of positive definite ones is denoted $\L_s^+(X,X^*)$. 
\end{definition}
\noindent Despite the analogy with Hilbert spaces the partial order on $\L_s(X,X^*)$ can behave quite differently. In particular, $\L_s^+(H)$ is generating and contains order units, while $\L_s^+(X,X^*)$ may not do either. Lin showed that a Banach space $X$ is isomorphic to a Hilbert space if and only if there is $P\in\L(X,X^*)$ satisfying $m\,\|x\|^2\leq\langle Px,x\rangle\leq M\,\|x\|^2$ for all $x\in X$ and some $m,M>0$ \cite[Thm.3]{Lin}. A corollary of his result is that for a reflexive $X$ any symmetric strictly positive definite $P\in\L_s^+(X,X^*)$ induces an isomorphism of $X$ with a Hilbert space, see Proposition 3.1 and Remark 3.2 in \cite{Dri}. On the other hand, if $P\in\L_s^+(X,X^*)$ is an order unit it will be strictly positive definite because $\langle Px,x\rangle=\llangle P,x\otimes x\rrangle>0$ for any $x\neq0$ since non-zero positive functionals are strictly positive on order units. Therefore, as long as $X$ is not isomorphic to a Hilbert space the cone $\L_s^+(X,X^*)$ does not contain order units. 

The authors of \cite{Kal} prove that $\L_s^+(X,X^*)$ is generating if and only if every $P\in\L_s(X,X^*)$ factors through a Hilbert space (by the reproducing kernel Hilbert space construction any $P\in\L_s^+(X,X^*)$ always so factors, see Section \ref{S5}). An example of non-factorizable $P\in\L_s(l_p,l_p^*)$ for $1<p<2$ is given in \cite{Sari}, hence 
$\L_s^+(l_p,l_p^*)$ is not generating for such $p$. Moreover, $\L_s^+(X,X^*)$ with any infinite-dimensional $X=L_p(\mu)$ and positive measure $\mu$ is not generating for $1<p<2$ because it contains a complemented copy of $l_p$, see 
\cite[Thm.3.3]{Kal}.

To use duality for $\L_s(X,X^*)$ we need symmetric tensor products rather than the general ones.
\begin{definition}\label{pinorm} Let $X$ be a Banach space. The symmetric algebraic tensor product $X\ohat X$ is the linear span in $X\otimes X$ of tensors of the form $x\otimes y+y\otimes x$, where $x,y\in X$. 
The {\sl symmetric projective tensor product $X\ohat_\pi X$} is the closure of $X\ohat X$ in $X\otimes_\pi X$.
\end{definition} 
\noindent The space $X\ohat_\pi X$ is canonically isomorphic to the Banach quotient of $X\otimes_\pi X$ by the subspace of antisymmetric tensors spanned by $x\otimes y-y\otimes x$. The projective duality along with standard isomorphisms for the dual spaces to subspaces and quotients \cite[4.8]{Rud} implies that $\L_s(X,X^*)\simeq(X\ohat_\pi X)^*$ as Banach spaces. Since elements of $\L_s(X,X^*)$ annihilate antisymmetric tensors $\s(\L_s(X,X^*),X\ohat X)$ coincides with the weak operator topology on $\L_s(X,X^*)$, and $\s(\L_s(X,X^*),X\ohat_\pi X)$ is the weak* topology. By inspection, Theorem \ref{C0pi} remains valid for restrictions of $\T(t)$ and $\T_*(t)$ to $\L_s(X,X^*)$ and $X\ohat_\pi X$ respectively. When $V(t)=T(t)$ the implemented semigroups restricted to $\L_s(X,X^*)$ turn out to be positive, indeed 
$\langle\bigl(\T(t)P\bigr)x,x\rangle=\langle PT(t)x,T(t)x\rangle\geq0$ for positive definite $P$.
\begin{definition} 
Let $X$ be a reflexive Banach space and $T(t)$ be a $C_0$ semigroup on it with the generator $A$. 
The positive semigroup $\T(t)P:=T^*(t)\,P\,T(t)$ on $\L_s(X,X^*)$ will be called the {\sl Lyapunov semigroup} of $T(t)$, and its generator $\A$ will be called the {\sl Lyapunov generator}.
\end{definition}
\noindent When $X$ is a Hilbert space $\T(t)$ is considered in \cite[I.3.16]{EN}, \cite{GN} and \cite[3.4]{Kuh0}. It 
is $C_0$ if only if $T(t)$ is uniformly continuous. The next theorem is the main result of this section.
\begin{theorem}\label{CbLyap} Let $X$ be a Banach space, $T(t)$ be a $C_0$ semigroup on it and $X\ohat_\pi X$ be its symmetric projective tensor product with itself. Then the positive cone $(X\ohat_\pi X)^+$ dual to $\L_s^+(X,X^*)$ is generating, and the Lyapunov semigroup $\T(t)$ is a $\pi$-compatible positive $C_b$ semigroup in the weak* topology.
\end{theorem}
\begin{proof} By the Grothendieck representation theorem \cite[III.6.4]{Sch}, any $\rho\in X\otimes_\pi X$
is of the form $\rho=\sum_{n=1}^\infty a_n\,x_n\!\otimes y_n$, where $x_n,y_n\in X$ are bounded sequences and the numbers $a_n$ satisfy $\sum_{n=1}^\infty|a_n|<\infty$. By polarization, for $\rho\in X\ohat_\pi X$ one can take 
$y_n=x_n$. Set $\rho^+:=\sum_{a_n\geq0} a_n\,x_n\!\otimes x_n$, $\rho^-:=-\sum_{a_n\leq0}a_n\,x_n\!\otimes x_n$, then
$\rho^\pm\in(X\ohat_\pi X)^+$ and $\rho=\rho^+-\rho^-$, so the positive cone is generating.

The weak* continuity is a part of Theorem \ref{C0pi}, the locally uniform boundedness follows from the norm estimate 
$\|\T(t)P\|=\|T^*(t)PT(t)\|\leq\|T(t)\|^2\|P\|$. Also by Theorem \ref{C0pi} the adjoint semigroup $\T'(t)=\T_*(t)$ is a 
$C_0$ semigroup on $X\ohat_\pi X$, so $t\mapsto\pi(\T_*(t)\rho)$ is continuous. Thus, $\T(t)$ is $\pi$-compatible.
\end{proof}
Note that we did not assume that $X$ is reflexive in Theorem \ref{CbLyap}, this assumption is needed to apply it to the Lyapunov equations. By a result of Freeman \cite[Thm.4]{Fr} the Lyapunov generator in the weak operator topology has the domain
$$
\D_{\A}=\Big\{P\in\L_s(X,X^*)\,\Big|\,P\D_A\subseteq\D_{A^*}\,,\|A^*P+PA\|<\infty\Big\}\,,
$$
and is given explicitly by $\A P=\overline{A^*P+PA}$, where the bar denotes the unique extension from $\D_{A}$ to $X$. Since the generator is determined by convergence of sequences and the weak* topology coincides with the weak operator topology on convergent sequences, the weak* generator is the same.

\section{Lyapunov equations in control theory}\label{S5}

In this section we interpret the notions of stability, detectors and observers for the Lyapunov semigroups over  reflexive Banach spaces, and apply to them the abstract theorems of Section \ref{S3}. This leads in particular to a generalization of Wohnam's Lyapunov theorem for the Lyapunov equations with detectable right hand sides. We also derive some practical tests to verify detectability.

For the dual pair $(\L_s(X,X^*),X\ohat X)$ weak $L^1$ stability amounts to:
\begin{equation}\label{OpWstab}
\int_0^\infty\llangle x\otimes x,\T(t)P\rrangle\,dt=\int_0^\infty\langle T^*(t)\,P\,T(t)x,x\rangle\,dt
=\int_0^\infty\langle PT(t)x,T(t)x\rangle\,dt<\infty
\end{equation}
for all $x\in X$ and $P\in \L_s^+(X,X^*)$. If $X=X^*=H$ is a Hilbert space it suffices that \eqref{OpWstab}
holds with $P=I$ since $I$ is an order unit, reducing \eqref{OpWstab} to $\int_0^\infty\|T(t)x\|^2\,dt<\infty$. By the Datko-Pazy theorem \cite[II.1.2.2]{Ben}, \cite[Prop.9.4]{Clem} this is equivalent to $T(t)$ being exponentially stable. In the reflexive Banach case there is no analog of $I$ since $X$, $X^*$ are not isomorphic, and the requirement that \eqref{OpWstab} holds for all positive definite $P$ is hard to verify. It is only obvious that it is stronger than weak $L^2$ stability of $T(t)$ which one gets by taking $P$ of the form $Px:=\langle\pfi,x\rangle\pfi$ with $\pfi\in X^*$.

Weak $L^1$ detectors are interpreted analogously, $Q\in \L_s^+(X,X^*)$ is one if for every $x\in X$\,:
\begin{equation}\label{OpWdet}
\int_0^\infty\langle QT(t)x,T(t)x\rangle\,dt<\infty
\implies\int_0^\infty\langle PT(t)x,T(t)x\rangle\,dt<\infty\text{ for all } P\in \L_s^+(X,X^*)\,.
\end{equation}
Again, for Hilbert spaces $P=I$ suffices and \eqref{OpWdet} can be rewritten as 
\begin{equation}\label{OpWdetH}
\int_0^\infty\langle\|Q^{\frac12}T(t)x\|^2\,dt<\infty
\implies\int_0^\infty\|T(t)x\|^2\,dt<\infty\text{ for all } P\in \L_s^+(X,X^*)\,.
\end{equation}
If $Q=I$ this is vacuously satisfied, so every $C_0$ semigroup on a Hilbert space admits a detector. In reflexive Banach spaces existence of detectors for every Lyapunov semigroup is an open question. 

As one can see, weak $L^1$ concepts are technically inconvenient beyond Hilbert spaces. We shall see that $\pi$-stability and $\pi$-detectors provide a better framework when $\pi$ is the projective tensor norm. We now interpret the $\pi$--norm induced concepts for the Lyapunov semigroups $\T(t)$. Recall that $\T(t)$ is $L^1$ $\pi$-stable if $\int_0^\infty\pi\bigr(\T_*(t)\rho\bigl)\,dt<\infty$. Taking $\rho=x\otimes x$ we see that the cone $\F:=(X\ohat X)^+$ is a $\T_*$--invariant stability subset of $(X\ohat_\pi X)^+$ for any Lyapunov semigroup since $\T_*(t)(x\otimes x)=T(t)x\otimes T(t)x$. Suppose $\T(t)$ is $L_1$ $\pi$--stable on $(X\ohat X)^+$, then
$$
\int_0^\infty\pi\bigr(\T_*(t)(x\otimes x)\bigl)\,dt=\int_0^\infty\pi\Bigr(T(t)x\otimes T(t)x\Bigl)\,dt
=\int_0^\infty\|T(t)x\|^2\,dt<\infty
$$
for all $x\in X$. By the Datko-Pazy theorem $T(t)$ is exponentially stable in the Banach norm, and hence $\T(t)$ is $L_1$ $\pi$--stable. This also shows that $L^1$ $\pi$-stability is equivalent to the exponential stability for the Lyapunov semigroups. 

Recall that an operator $Q\in\L_s^+(X,X^*)$ is an $L^1$ $\pi$-detector relative to $(X\ohat X)^+$ if 
for all $\rho\in(X\ohat X)^+$:
$$
\int_0^\infty\llangle\rho,\T(t)Q\rrangle\,dt<\infty\implies\int_0^\infty\pi\bigr(\T_*(t)\rho\bigl)\,dt<\infty\,.
$$
This condition holds if and only if it holds on monomials giving
$$
\int_0^\infty\langle QT(t)x,T(t)x\rangle\,dt<\infty\implies\int_0^\infty\|T(t)x\|^2\,dt<\infty\,.
$$
To make the last condition look more recognizable let us assume that there is a Hilbert space $\mathcal{N}$ and a bounded linear operator 
$C:X\to\mathcal{N}$ such that $Q=C^*C$. Then the last condition turns into
$$
\int_0^\infty\|CT(t)x\|^2\,dt<\infty\implies\int_0^\infty\|T(t)x\|^2\,dt<\infty\,,
$$
which is more suggestive from the control theory point of view. If $X$ is finite dimensional one can equivalently ask that $CT(t)x\to0$ imply $T(t)x\to0$ when $t\to\infty$ for all $x\in X$, which is a standard definition 
of detectability \cite[3.6]{Wh}. We are led to the following infinite dimensional generalization, cf. \cite[II.1.2.2]{Ben}.
\begin{definition}\label{detL2} Let $X$ be a reflexive Banach space and $T(t)$ be a $C_0$ semigroup on it 
with the generator $A$. Let $\mathcal{N}$ be a Hilbert space and $C:X\to\mathcal{N}$ be a bounded operator. We say that the pair {\sl $(C,A)$ is detectable in $L^2$} if for all $x\in X$ 
$$
\int_0^\infty\|CT(t)x\|^2\,dt<\infty\implies\int_0^\infty\|T(t)x\|^2\,dt<\infty\,.
$$
\end{definition}
\noindent We now turn to the final $\pi$-observers at time $t_0$. Definition \ref{dpiobs} with $\F=(X\ohat X)^+$ requires existence of an $\e>0$ such that for all $\rho\in(X\ohat X)^+$ we have 
$\int_0^{t_0}\llangle\rho,\T(t)Q\rrangle\,dt\geq\e\,\pi\bigr(\T_*(t)\rho\bigl)$. For $Q=C^*C$ as above this reduces to 
the usual continuous final observability \cite[3.23]{CR}.
\begin{definition}
Let $X$ be a reflexive Banach space and $T(t)$ be a $C_0$ semigroup on it 
with the generator $A$. Let $\mathcal{N}$ be a Hilbert space and $C:X\to\mathcal{N}$ be a bounded operator. 
We say that the pair {\sl $(C,A)$ is continuously finally observable on $[0,t_0]$} if there exists an $\e>0$ 
such that for all $x\in X$
$$
\int_0^{t_0}\|CT(t)x\|^2\,dt\geq\e\,\|T(t_0)x\|^2\,.
$$
\end{definition}
Having $Q$ represented as $Q=C^*C$ does not restrict generality. Given $Q\in\L_s^+(X,X^*)$ define an inner product on $\Ran Q\subseteq X^*$ by $(Qx,Qy):=\langle Qx,y\rangle$, it is well defined since $Q^*=Q$. Denote by 
$H_Q$ the completion of $\Ran Q$ in the inner product norm. There is a natural embedding $i_Q:H_Q\to X^*$ and one easily verifies that $Q=i_Qi_Q^*$. Therefore, if we set $\mathcal{N}=H_Q$ and $C=i_Q^*$ then $Q=C^*C$. The space $H_Q$ is known as the {\sl reproducing kernel Hilbert space} of $Q$, and the construction of $C$ replaces the positive square root construction in Hilbert spaces. Note that this shows in particular that any $Q\in\L_s^+(X,X^*)$ factors through a Hilbert space. 
When $X$ is a Hilbert space identified with its dual, $i_Q$ can be identified with $
Q^{\frac12}$ and $H_Q$ can be identified with $\Ran Q^{\frac12}$. Reproducing kernel Hilbert spaces were used to study the Lyapunov equations in \cite{GvN,vN2}. We now state our main result on the Lyapunov equations over reflexive Banach spaces.
\begin{theorem}\label{Wonh} Let $X$ be a reflexive Banach space and $T(t)$ be a $C_0$ semigroup on it 
with the generator $A$. Let $\mathcal{N}$ be a Hilbert space and $C:X\to\mathcal{N}$ be a bounded operator such that the pair $(C,A)$ is detectable in $L^2$. Then the following conditions are equivalent:

{\rm(i)} $\overline{A^*P+PA}=-C^*C$ has a positive definite solution $P\in\L_s(X,X^*)^+$ such that 
$P(\D_A)\subseteq\D_{A^*}$ (the bar stands for the unique extension from $\D_A$ to $X$);

{\rm(ii)} $T(t)$ is exponentially stable;

{\rm(iii)} Lyapunov generator $\A P:=\overline{A^*P+PA}$ has a bounded inverse on 
$\L_s(X,X^*)$\\ and $-\A^{-1}\geq0$.
\end{theorem}
\begin{proof} By Theorem \ref{CbLyap} with $\X=\L_s(X,X^*)$, $\T(t)$, $z=C^*C$ and $\F=(X\ohat X)^+$ satisfy the conditions of Theorem \ref{Lyappi}. We see that Theorem \ref{Lyappi}(i) is equivalent to the (i) here since $\pi$ is the induced operator norm on $\L_s(X,X^*)$. Claims (ii),(iii) of Theorems \ref{Lyappi} are equivalent to (ii),(iii) here as well, and it remains to note that exponential stability of $\T(t)$ is equivalent to that of $T(t)$.
\end{proof}
\noindent This is an infinite-dimensional generalization of a theorem of Wonham \cite[12.4]{Wh}, 
the Hilbert space version is due to Zabczyk \cite[Lem.3]{Zb0}, see also \cite[Thm.I.1.2.4]{Ben}, \cite[4.2]{GN} when $C^*C$ is the identity operator. The result appears to be new for Banach spaces. In \cite[Thm.4.4]{GvN} the Lyapunov equation is considered in a general Banach setting, but it is only proved that (i) is equivalent to $\int_0^t\T(t)(C^*C)\,dt$ having a weak limit when $t\to\infty$. We were able to go further by exploiting the projective duality and compatibility with the projective tensor norm. 

Of course, the theorem is only useful if one has a way of verifying the detectability condition. 
The next $L^2$ detectability test is a direct consequence of Theorem \ref{piobs}, and the continuous final observability can often be established in applications via a priori estimates, see Example 3.27 and references to Chapter 3 in \cite{CR}.
\begin{theorem}\label{Cobs} Let $X$ be a reflexive Banach space and $A$ be the generator of a positive $C_0$ semigroup on it. Let $\mathcal{N}$ be a Hilbert space and $C:X\to\mathcal{N}$ be a bounded operator. If the pair $(C,A)$ 
is continuously finally observable at time $t_0>0$ then it is detectable in $L^2$. 
\end{theorem}
\noindent Finally, we give another sufficient condition for $L^2$ detectability specific to the Lyapunov semigroups, the one used by Wonham in the finite dimensional case, and by Zabczyk in the Hilbert case. 
We need a definition first.
\begin{definition} In the notation of Definition \ref{detL2} we say that a pair {\sl $(C,A)$ is exponentially detectable} if there exists a bounded operator $F:\mathcal{N}\to X$ such that $A-FC$ generates an exponentially stable $C_0$ semigroup.
\end{definition}
\begin{theorem}\label{ExpDet} If $(C,A)$ is exponentially detectable then it is detectable in $L^2$.
\end{theorem}
\begin{proof} The proof follows the idea from \cite[Lem.3]{Zb0}, see also \cite[Thm.I.1.2.6]{Ben}. We have to show that for any $x\in X$
$$
\int_0^\infty\|CT(t)x\|^2\,dt<\infty
\implies\int_0^\infty\|T(t)x\|^2\,dt<\infty\,, 
$$
where $T(t)$ is the semigroup generated by $A$. By exponential detectability, there is a bounded $F$ such that $A-FC$ generates an exponentially stable $C_0$ semigroup $T_{A-FC}(t)$. Since $A=(A-FC)+FC$ we have by the variation of parameters formula
$$
T(t)x=T_{A-FC}(t)x+\int_0^tT_{A-FC}(t-s)\,FC\,T_{A-FC}(s)x\,ds\,.
$$
Passing to norms,
$$
\|T(t)x\|\leq\|T_{A-FC}(t)x\|+\|F\|\int_0^t\|T_{A-FC}(t-s)\|\,\|C\,T_{A-FC}(s)x\|\,ds<\infty\,.
$$
The last term is a convolution of scalar functions supported on $[0,\infty)$. By the Young inequality for 
convolutions,  $\|f*g\|_{L_r}\leq\|f\|_{L_p}\,\|g\|_{L_q}$ whenever $\frac1p+\frac1q=1+\frac1r$, see \cite[IX.4]{RS2}.
Setting $f(s):=\|T_{A-FC}(s)\|$, $g(s):=\|C\,T_{A-FC}(s)x\|$ and applying the Young inequality with $p=1$ and 
$q=r=2$ we obtain
\begin{multline}
\Bigl(\int_0^\infty\|T(t)x\|^2\,dt\Bigr)^{\frac12}\leq
\Bigl(\int_0^\infty\|T_{A-FC}(t)x\|^2\,dt\Bigr)^{\frac12}\\
+\|F\|\int_0^\infty\|T_{A-FC}(t)x\|^2\,dt\,\Bigl(\int_0^\infty\|CT(t)x\|^2\,dt\Bigr)^{\frac12}\,.
\end{multline}
Since $T_{A-FC}(t)$ is exponentially stable and $\int_0^\infty\|CT(t)x\|^2\,dt<\infty$ by assumption we conclude that 
$\int_0^\infty\|T(t)x\|^2\,dt<\infty$, and hence $(C,A)$ is detectable in $L^2$.
\end{proof}

{\em Acknowledgements:} The author is grateful to his former advisor Yu.V. Bogdansky for introducing him to the subject of one-parameter semigroups and discussing many ideas of this work.

\end{document}